\documentclass[reqno]{amsart}
\RequirePackage{fix-cm}
\usepackage[frenchb,english]{babel}
\usepackage{graphicx}
\usepackage{subcaption}
\usepackage{mathabx}

\usepackage{stmaryrd}
\usepackage{verbatim}
\usepackage{appendix}
\usepackage{natbib}
\usepackage{siunitx}
\usepackage[margin=1.2in]{geometry}
\usepackage{multirow}
\newtheorem{theorem}{Theorem}[section]

\theoremstyle{definition}
\newtheorem{definition}[theorem]{Definition}

\newtheorem{prop}{Proposition}
\theoremstyle{remark}
\newtheorem{remark}[theorem]{Remark}

\numberwithin{equation}{section}
\newcommand{\norm}[1]{\left\lVert#1\right\rVert}
\newcommand{\ie}{\textit{i.e.,\ }}
\newcommand{\eg}{\textit{e.g.,\ }}

\newcommand{\dd}{\textrm{d}}

\usepackage{enumitem}   

\usepackage{rotating}
\usepackage{breakcites}

\usepackage[colorlinks,citecolor=blue]{hyperref}

\begin{document}

\title[Geometric Lagrangian Descriptor]{Geometric parametrisation of Lagrangian
	Descriptors for 1 degree-of-freedom systems}
\author{R\'emi P\'edenon-Orlanducci}
\address{ENSTA Paris, Institut Polytechnique de Paris, 91120 Palaiseau, France}
\email{remi.pedenon-orlanducci@ensta-paris.fr}

\author{Timoteo Carletti}
\address{naXys, Namur Institute for Complex Systems, University of Namur, rue de Bruxelles 61, B5000, Namur, Belgium}
\email{timoteo.carletti@unamur.be}

\author{Anne Lemaitre}
\address{naXys, Namur Institute for Complex Systems, University of Namur, rue de Bruxelles 61, B5000, Namur, Belgium}
\email{anne.lemaitre@unamur.be}

\author{J\'er\^ome Daquin}
\address{
naXys, Namur Institute for Complex Systems, University of Namur, rue de Bruxelles 61, B5000, Namur, Belgium.
School of Enginnering \& IT, University of New South Wales Canberra at ADFA, ACT 2600, Canberra, Australia}
\email{jerome.daquin@unamur.be}


\date{\today}

\date{\today}

\dedicatory{}
\maketitle
\begin{abstract}
	Lagrangian Descriptors (LDs) are scalar quantities able to reveal separatrices, manifolds of hyperbolic saddles, and chaotic seas of dynamical systems. 
	A popular version of the LDs consists
	in computing the arc-length of trajectories over a calibrated time-window. 
	Herein we introduce and exploit 
	an intrinsic geometrical parametrisation of LDs, free of 
	the time variable, for $1$ degree-of-freedom Hamiltonian systems. 
	The parametrisation depends solely on the energy of the system and on the geometry of the associated level curve. We discuss applications of this framework on classical 
	problems on the plane and cylinder, including the cat’s eye, $8$-shaped and fish-tail separatrices. 
	The developed apparatus 
	allows to characterise semi-analytically the rate at which 
	the derivatives of the geometrical LDs  become singular when approaching the separatrix. 
	For the problems considered, the same power laws of divergence are found irrespective from the dynamical system. 
	Some of our  results are connected with existing estimates obtained 
	with the temporal LDs 
	under approximations.  
	The  geometrical formalism provides alternative insights of the mechanisms driving this dynamical indicator. 
\end{abstract}

\section{Introduction}\label{sec:intro}
Over the last decade,  Lagrangian Descriptors (LDs)  
have proven their abilities to reveal the template
of dynamical systems \citep{jaMa09,cMe10}. They have been used successfully in a number of instances to portray  dynamical structures such as separatrices and manifolds of hyperbolic saddles \citep{amMa13,cLo17,rCr21}, manifolds of normally hyperbolic manifolds (NHIM) \citep{nSh19} or chaotic seas \citep{fRe19}. Whilst their initial development is rooted in geophysical flows, LDs expanded to discrete systems and data sets \citep{cMe14,cLo15,cGa20}, found applications in dynamical chemistry \citep{cGa17} or billiards dynamics \citep{cGa21}.    
In this contribution, we consider autonomous differential systems
\begin{align}\label{eq:ODE}
\dot{x}=f(x),\, x \in D \subset \mathbb{R}^{n}, 
\end{align} 
where the vector field $f \in C^{k}$, $k \ge 1$, with a particular emphasis to the case $n=2$. 
For Eq.\,(\ref{eq:ODE}), being provided an initial condition $x_{0}$ and a time $t >0$, the LD  is defined as the scalar  
\begin{align}\label{eq:LD}
\textrm{LD}(x_{0},t) = \int_{-t}^{t}
(g \circ h)
\big(
\phi^{s}(x_{0})
\big) \, \dd s,
\end{align}
where $\phi^{s}$ denotes the flow (hence the qualification ``Lagrangian'') at time $s$ associated to Eq.\,(\ref{eq:ODE}).
Common choices  of observables  
in the literature are $h=f$ (the vector field itself), and 
$g(u)=\norm{u}_{2}=\sqrt{\sum_{i=1}^{n}u_{i}^{2}}$ (see \eg \cite{jaMa09,cMe10}) or 
$g(u)=\sum_{i=1}^{n}\vert u_{i}\vert^{p}$ for $p \in (0,1]$ (see \eg \cite{cLo17}). 
The observables encapsulate a bounded, positive quantity
that is an intrinsic geometrical and/or physical property of the dynamical system
along a trajectory (over a finite time), as stated in \cite{amMa13}.
In the subsequent we consider $h=f$ and $g$ given by the Euclidean norm. Under this setting, Eq.\,(\ref{eq:LD}) becomes
\begin{align}\label{eq:LDlength}
\textrm{LD}(x_{0},t) = \int_{-t}^{t}
\norm{\dot{x}(s)}_{2}
\, \dd s,
\end{align}
and corresponds to the arc-length of the trajectory starting from $x_{0}$ and
computed over the time window $[-t,t]$. 

The ability of the LD method to reveal hyperbolic structures in the phase space relies on the loss of regularity of the LD metric when transversally crossing critical energy levels, \ie associated separatrices, a property initially refereed to as ``\textit{abrupt change}'' 
and later termed as ``\textit{singular feature of the LDs}'' (see \eg \cite{cLo17}). Besides heuristic arguments discussed in \cite{amMa13}, precise and rigorous analytical justifications have been presented on planar linear models, such as the linear saddle or rotated version of it \citep{amMa13,cLo17}. In essence, by taking advantage of the relatively simple and explicit form of the flows, the leading term of Eq.\,(\ref{eq:LDlength}) can be estimated explicitly by using an ad-hoc splitting of the time window $[-t,t]$ over which the Cauchy condition is advected. The loss of regularity of the LD metric when crossing separatrices can then be established  from this estimation. 
The dependence upon the choice of  $t$ will be discussed in the following. 

For higher dimensional dynamical systems, for example 
in the framework of nearly integrable Hamiltonian systems supporting separatrix splitting, in order to portray finite pieces of the manifolds, it is customary to separate Eq.\,(\ref{eq:LD}) as
\begin{align}
\textrm{LD}(x_{0},t) = \textrm{LD}^{+}(x_{0},t)+
\textrm{LD}^{-}(x_{0},t),
\end{align} 
where 
\begin{align}
\textrm{LD}^{+}(x_{0},t) = \int_{0}^{t} 
\norm{\dot{x}(s)}_{2}
\, \dd s,
\end{align}
and 
\begin{align}
\textrm{LD}^{-}(x_{0},t)=\textrm{LD}(x_{0},t)
-
\textrm{LD}^{+}(x_{0},t),
\end{align}
in order to delineate respectively the 
stable and unstable manifolds associated to the hyperbolic invariant (\eg a saddle fixed point, or higher dimensional hyperbolic invariant like a  NHIM). 
As customary with dynamical indicators, there are no general rules to determine the time window over which the method succeeds. 
The trade-off between the computational burden (large final time $t$) and the ability of the method to detect quickly structures (small time regime) is calibrated by trial and error or saturation checks,  
even though the knowledge of specific timescales (\eg $e$-folding time, period of periodic orbits) or physical constraints might guide  the numerical experiments. 
It is worthwhile to underline the absence of need of  variational dynamics associated to a deviation vector in the derivation of the LDs (as it is the case for Lyapunov exponents or other variational methods, like the Fast Lyapunov Indicator or the mean exponential growth of nearby orbits, see \cite{chSk10} for a general review).  
In particular, 
Eq.\,(\ref{eq:LD}) is appealing from the numerical standpoint and constitutes an orbit-based only diagnostic. As we already mentioned, and contrarily to  variational methods,
the LDs do not discriminate hyperbolic trajectories through the final value of the indicator. Instead, and rather similarly to the frequency analysis of dynamical system method \citep{jLa93}, the \textit{regularity} of the LD application is central to obtain the global picture of the system. The hyperbolic structures are then located or portrayed by studying the regularity of the LDs computed in lower dimensional space, \eg by restricting the initial condition to a domain $\mathcal{D} \subset \mathbb{R}$ (a one-dimensional map, also called a landscape) or 
to a subset $\mathcal{D} \subset \mathbb{R}^{2}$ (map).
A heat map can be computed to visualise the results,
similarly to the method of painting the energy integral over the phase space \citep{cSh90}.\\

In the following, we complement the theory of LDs for conservative $1$ degree-of-freedom (DoF) problems having one hyperbolic saddle point and supporting different separatrices topologies. The paper is outlined as follows:
\begin{itemize}
	\item In Section \ref{sec:pendulum}, we use the framework of the pendulum model to introduce 
	an intrinsic and time-free parametrisation of the LD. The parametrisation $\ell_{E}$, corresponding to  the length of an orbit in the phase space, 
	depends solely on the specified energy $E$ and the geometry of the associated level curve. Some properties of this observable are studied and discussed against their temporal counterpart. We highlight that this indicator succeeds in detecting the separatrix. The new definition allows to develop a convenient semi-analytical apparatus useful to characterize the rate at which $\vert \dd \ell_{E}/\dd E\vert$ becomes singular when $E \to E_{\textrm{sx}}$, the energy labeling the system separatrix.
	\item In Sections \ref{sec:Duffing} and \ref{sec:FishTail}, the previous steps are repeated on Hamiltonian models supporting respectively an $8$-shaped 
	and fish-tail separatrix.  
	Despite the different topologies of those critical curves, 
	the same laws of divergence for $\vert \dd \ell_{E}/\dd E\vert$ are revealed.  
\end{itemize}
We close the paper by summarising our conclusions. 

\section{Revisiting  Lagrangian Descriptors on the pendulum}\label{sec:pendulum} 
We start by revisiting the LDs  on the pendulum model  
whose Hamiltonian is 
\begin{align}\label{eq:pendulum}
\mathcal{H}(r,\theta)=\frac{r^{2}}{2}
-\cos \theta -1, \,(r,\theta) \in \mathcal{D}=\mathbb{R}\times [0,2\pi].
\end{align}
The constant term $-1$ in Eq.\,(\ref{eq:pendulum}) ensures that $E=0$  on the separatrix.
The phase space contains two equilibrium points, one elliptic at the origin  
$(0,0)$ and one hyperbolic  at $(0,\pi)$. 
The phase space is organised by librational ($E<0$) and circulational ($E>0$) motions separated by the separatrix $(E=0)$. 
Moving along the line $\theta=0$ for  increasing values of $r \ge 0$, we reach the apex of the separatrix 
by solving for $\delta r$ the equation 
\begin{align}
\mathcal{H}(\delta r,0)=-2,
\end{align}
\ie for $\delta r=2$.
The full aperture of the cat-eye separatrix has thus the width $\Delta r = 2\delta r = 4$.  
We refer the reader to Fig.\,\ref{fig:pendulum} for a visualisation of the phase space of this model obtained by using the level-set method.
For a generic point $x_{0}=(r_{0},\theta_{0})$, trying to estimate the LD given in Eq.\,(\ref{eq:LD}) is cumbersome and requires the knowledge of the flow and to manipulate  elliptic functions. 
On the other hand, for $1$-DoF systems, as orbits coincide with level curves, an approximate value of their lengths in the phase space is rather easy to pin down by visual inspection.   
This observation leads us to introduce a more  geometric version of the LD defined as follow. 
\begin{definition}[Geometrical Lagrangian Descriptor]
	Given an energy level $E$, the geometrical LD, denoted  $\ell(E)$, is the scalar corresponding to  the length of the level curve  $\mathcal{H}(r,\theta)=E$. 
\end{definition} 
Let us emphasise that working with the energy level, de facto accounts to deal with
an infinitely large time interval and thus by removing any time dependence of $\ell(E)$.
This definition is general enough and is discussed further in the two remarks below. 
For the pendulum model of Eq.\,(\ref{eq:pendulum}), 
note that for
$E < \min_{\theta} (-\cos \theta - 1)=-2$, the set of level curves is empty and we thus consider the former definition for $E \ge -2$ only.  
Since motions are bounded,  we have $\ell(E) \in \mathbb{R}^{+}$.  
From the symmetry of the phase space, for $r \ge 0$, one can interpret the level curve $\mathcal{H}(r,\theta)=E$ as the planar curve, parametrised by $\theta$, given by
\begin{align} 
r(\theta;E)=\sqrt{2(E + \cos \theta +1)}. 
\end{align}
By exploiting the formula for the length of a curve, we thus might write    
$\ell(E)=2\tilde{\ell}(E)$ with
\begin{align}\label{eq:lE}
\tilde{\ell}(E)
=
\int_{\mathcal{D}_{E}} 
\sqrt{1+\Big(\frac{\dd r}{\dd \theta}\Big)^2}\,
\dd \theta,
\end{align}
and, after some straightforward algebraic computations, one gets 
\begin{align}\label{eq:lePend}  
\tilde{\ell}(E)=\int_{\mathcal{D}_{E}} 
\sqrt{1 + \frac{\sin ^2 \theta}{2(E + \cos \theta +1)}}
\,
\dd \theta. 
\end{align}      
The domains of integration are respectively defined as
\begin{align}
\left\{
\begin{aligned}
& \mathcal{D}_{E} = (\arccos(-E-1),-\arccos(-E-1)), \,\textrm{if} \ E < 0,  \notag \\
& \mathcal{D}_{E} = [-\pi,\pi], \textrm{otherwise}. 
\end{aligned}
\right.
\end{align}

\begin{figure}
	\centering
	\includegraphics[width=0.5\linewidth]{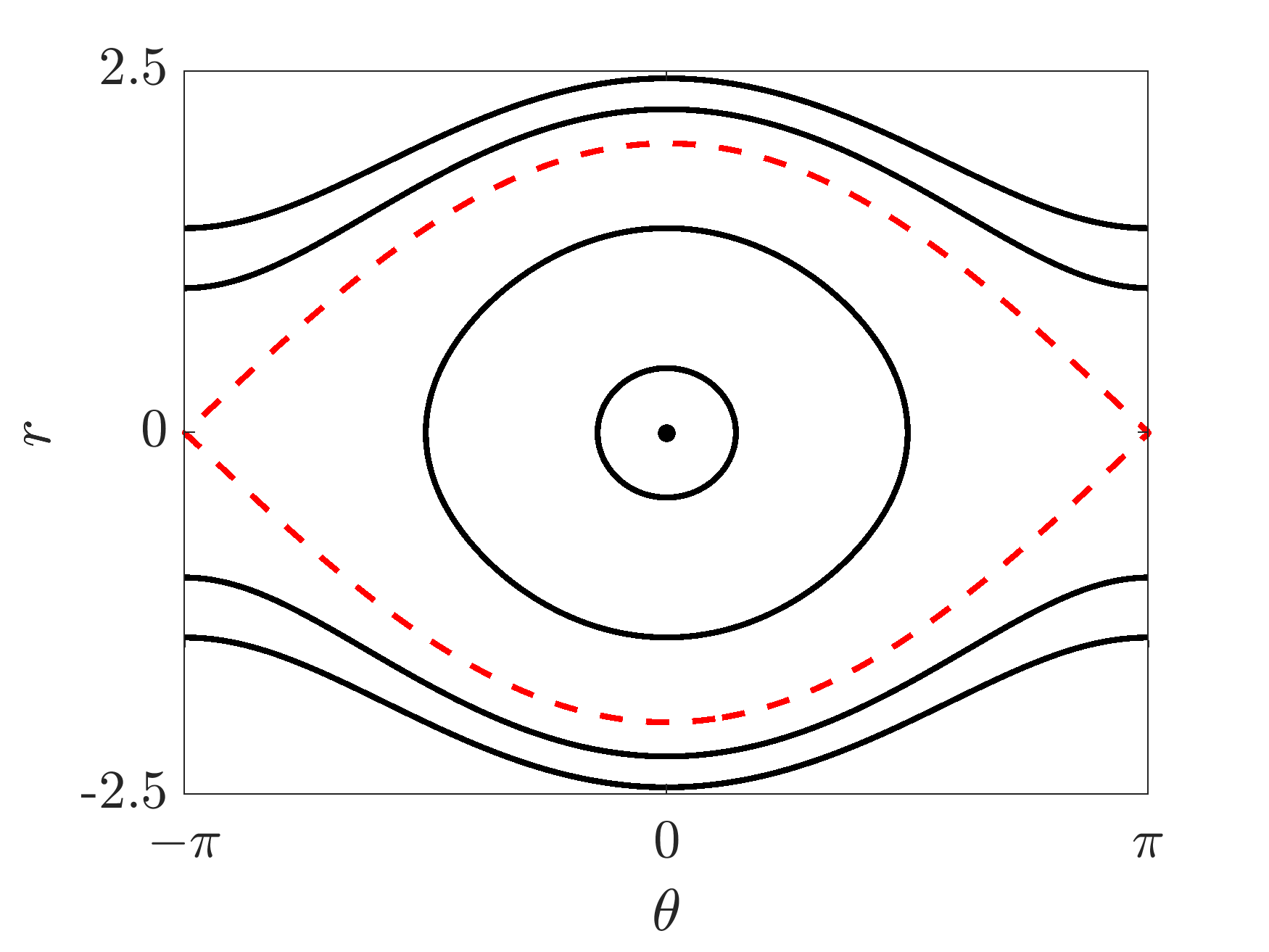}
	\caption{Phase space of the pendulum supporting the cat's eye separatrix (dashed  curve). For $1$-DoF system, the phase space approach allows to introduce a geometrical LD, a natural counterpart of the temporal LD but free of the time variable. The geometrical LD corresponds to the length $\ell(E)$ of the curve associated to the energy level $E$.}
	\label{fig:pendulum}
\end{figure}

We refer to the LD quantity defined in Eq.\,(\ref{eq:LD}) as \textit{temporal} LD, to emphasize the difference with the \textit{geometrical} LD 
just introduced\footnote{
	Let us remark that we refereed to the quantity $\ell(E)$ as `geometrical Lagrangian Descriptor' even if $\ell(E)$ has no Lagrangian nature as it does not rely on the flow.}. 
The geometrical LD has some benefits over the temporal formulation. 
Firstly, the expression does not rely on the time and the initial Cauchy condition $x_{0}$, but solely on the energy $E$. 
Secondly, its analytical expression can be derived easily through algebraic manipulators and numerically 
estimated. 
In particular, this last step does not rely on the knowledge of the flow (or any approximation of it, \eg either locally through linearisation or numerical approximation through solvers). 
Even though the geometrical LD we just introduced resemble its temporal counterpart, in the sense that both quantities are based on the length of orbits, properties of  $\ell(E)$ have not been investigated in the LD literature. To this point we now turn our efforts.

\begin{remark}
	[Geometrical LD for conservative 1-DoF systems]
	The definition of  $\ell(E)$ introduced on the pendulum model is general.
	For example, let us consider a mechanical system for which Newton's equations are 
	\begin{align}\label{eq:NewtonMech}
	\ddot{x}=-\nabla_{x}V(x), \, x \in \mathcal{I}\subset \mathbb{R},
	\end{align} 
	where $V(x)$ denotes a potential energy. 
	The Eq.\,(\ref{eq:NewtonMech}) admits the Hamiltonian formulation
	\begin{align}
	\mathcal{H}(x,v)=\frac{v^{2}}{2} + V(x),
	\end{align}
	where $v=\dot{x}$. 
	For $E \ge \min_{x}V(x)$, the two branches 
	of the level curves associated to the energy level $E=\mathcal{H}(x,v)$ are
	given by $v_{\pm}=\pm\sqrt{2(E-V(x))}$.
	The formula given in Eq.\,(\ref{eq:lE}) becomes
	\begin{align}
	\ell(E)=\int
	\sqrt{1 + 
		\frac{(\nabla_{x}V(x)^2)}{2(E-V(x))}}\,
	\dd x, 
	\end{align} 
	where the integral is computed over a suitable range of the variable $x$.
\end{remark}

\begin{remark}[First integral]
	The mechanical systems we discussed have the form ``kinetic'' plus ``potential'' energy and the total energy is preserved.    
	If more generally Eq.\,(\ref{eq:ODE}) possesses a first integral $I$, \ie a function $I: \mathcal{D} \to \mathbb{R}$ which satisfies
	\begin{align}
	\forall t, \, \forall x \in \mathcal{D}, \,
	I\big(\phi^{t}(x)\big)=I(x),
	\end{align}
	then, in the case $n=2$, orbits are constrained on the curves which are the level curves of the first integral $I$. The former construction can then be repeated.  
\end{remark}	

\subsection{Some properties of $\ell(E)$}   
The numerical evaluation of the continuous function  $E \mapsto \ell(E)$  over the domain $E \in [-2,1]$ is shown in the left panel of Fig.\,\ref{fig:leldpendulum}. The graph highlights that  
i) the length reaches a maximum at $E=0$ (the energy labeling the separatrix) and  
ii) the function $E \mapsto \ell(E)$ is not differentiable at $E=0$ (cusp point). The first statement can be proven rigorously.

\begin{prop}\label{prop:lEpendulum}
	Consider the pendulum model of Eq.\,(\ref{eq:pendulum}) 
	and the lengths  $\ell(E)$ for $E \ge -2$. 
	The value $\ell(E)$ is maximal on the separatrix,  \ie for $E=0$.  
\end{prop}	    

\begin{proof}
	For $E \ge 0$, whatever the values of $\theta \in (-\pi,\pi)$,   the  inequality 
	\begin{align}
	\frac{\sin^2 \theta}{2(\cos \theta + E +1)} \le \frac{\sin^2 \theta}{2(\cos \theta  +1)},
	\end{align}
	hold true. 
	The inequality is preserved by integrating both expressions over $(-\pi,\pi)$ 
	from which we derive $\ell(E) \le  \ell(0)$. 
	
	Let us now consider the case $-2\leq E <0$. By exploiting the symmetries of the pendulum, the total length $\ell(E)$ is four times the length of the portion of orbit belonging to the first quadrant, say $\ell_1(E)$; our claim will thus be proved once we will prove $\ell_1(E_1)<\ell_1(E_2)$ for all $-2\leq E_1<E_2<0$. Let us thus fix a positive real $\lambda$, consider the straight line $p=\lambda q$ and its intersection with the level curves with energy $-2\leq E <0$. One can implicitly express the intersection point $q(E)$ as a function of $E$, by considering $\lambda$ a fixed parameter, that is $\lambda^2 q^2=2(E+1+\cos q)$, and thus by geometrical considerations, to conclude that $q(E)$ is an increasing function of $E$. Moreover $q(E)\rightarrow 0$ for $E\rightarrow -2^+$ and $q(E)\rightarrow \pi$ for $E\rightarrow 0^-$. Let us now consider two values of energy $E_1$ and $E_2$, such that $-2\leq E_1<E_2<0$, and let us call $(p_1,q_1)$ and $(p_2,q_2)$ the intersections points of the straight line with the two energy levels (see Fig.~\ref{fig:FigTeo}). 
	
	\begin{figure}
		\centering
		\includegraphics[width=0.7\linewidth]{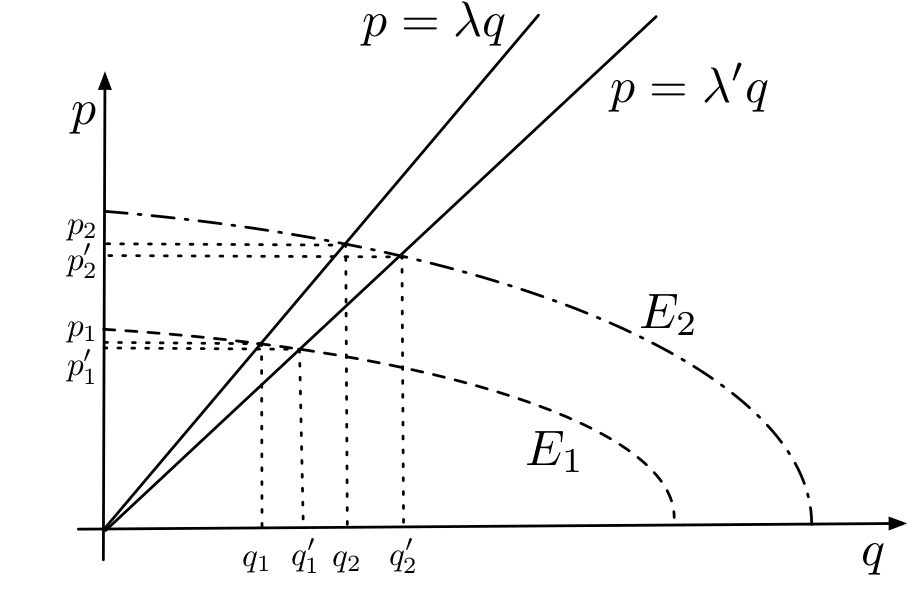}
		\caption{
			Orbits geometry used in the proof of Propostion~\ref{prop:lEpendulum} in the case $-2\leq E <0$. We reports two level curves associated to two values of the energy $-2\leq E_1<E_2<0$.
		}
		\label{fig:FigTeo}
	\end{figure}
	
	Let us now fix a second positive real $\lambda^\prime$ and consider again the intersections of the straight line $p=\lambda^\prime q$ with the two energy levels; let us denote them by $(p^\prime_1,q^\prime_1)$ and $(p^\prime_2,q^\prime_2)$. If $\lambda$ and $\lambda^\prime$ are infinitesimally close, say $\lambda^\prime=\lambda+d\lambda$, then (locally) the arc of level curves can be approximated by 
	\begin{align}\label{eq:delle1}
	d\ell_1(E_i)=\sqrt{1+\frac{\sin^2 q_i}{p^2_i}}dq_i,
	\end{align}
	where $dq_i=q_i-q_i^\prime$, the reason being that the level curves are smooth functions and thus we are approximating the length $\ell_1(E_i)$, \ie, the integral in Eq.\,(\ref{eq:lePend}), with the length of the segment tangent to the curve at $(p_i,q_i)$. Our goal is to prove that $d\ell_1(E_i)$ is an increasing function of the energy and from this to conclude that the same property holds true for the macroscopic arc length.
	
	To achieve our goal we have to express $dq_1$ in terms of $dq_2$. The first step is to relate $dp_1$, $dq_1$ and $d\lambda$. By exploiting the linear relation existing between $p_1$ and $q_1$ we get
	\begin{align}
	dp_1 = \lambda dq_1+q_1d\lambda ,
	\end{align} 
	and by recalling that $(p_1,q_1)$ belongs to the level curve we also get
	\begin{align}
	dp_1 = -\frac{\sin q_1}{p_1}  dq_1.
	\end{align} 
	Hence
	\begin{align}
	d\lambda = -\left(\frac{\sin q_1}{p_1} +\lambda\right) \frac{dq_1}{q_1}.
	\end{align} 
	
	Similar relations hold true for $dp_2$, $dq_2$ and $d\lambda$, we can then eliminate the dependence from $d\lambda$ and obtain
	\begin{align}
	\frac{dq_2}{q_2}\left(\frac{\sin q_2}{p_2} +\lambda\right) =  \left(\frac{\sin q_1}{p_1} +\lambda\right) \frac{dq_1}{q_1}.
	\end{align} 
	Recalling the expression~Eq.\,(\ref{eq:delle1}) for $d\ell_1(E_i)$, after some algebraic manipulations, we can eventually get
	\begin{align}
	\frac{d\ell_1(E_1)}{d\ell_1(E_2)}= \frac{\sqrt{q_1^2\lambda^2+\sin^2q_1}}{\lambda^2q_1+\sin q_1} q_1 \frac{\lambda^2q_2+\sin q_2}{\sqrt{q_2^2\lambda^2+\sin^2q_2}}\frac{1}{q_2}\equiv F_\lambda(q_1)\frac{1}{F_\lambda(q_2)},
	\end{align} 
	where the function $F_\lambda(q)$ has been defined by this last equality.
	
	Because of the above mentioned properties of $q_i$ as function of $E_i$, we can easily show that $F_\lambda\rightarrow 0$ for $q\rightarrow 0$ and $F_\lambda(\pi)=\pi/\lambda$.  By observing the numerical behaviour of the function $F_\lambda(q)$ (see Fig.~\ref{fig:FigTeo2}) one can conclude that it is monotone increasing for $q\in[0,\pi]$.
	
	\begin{figure}
		\centering
		\includegraphics[width=0.75\linewidth]{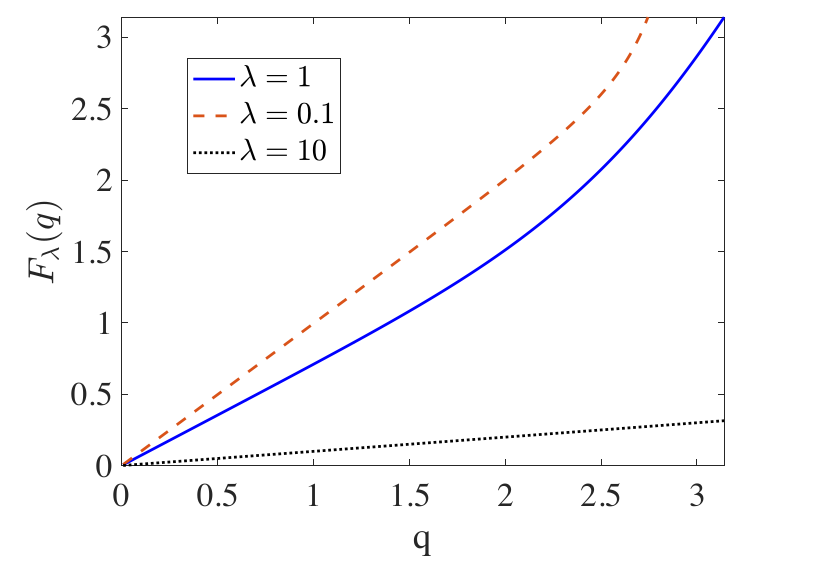}
		\caption{
			Graph of the function $F_\lambda(q)$ for three values of $\lambda$, $0.1$ (red dashed line), $1$ (blue solid line) and $10$ (black dotted line). One can appreciate the monotone increasing behaviour of the function in the domain $[0,\pi]$.
		}
		\label{fig:FigTeo2}
	\end{figure}
	By summarising the last result concerning $F_\lambda(q)$ together with the previous one stating that $q_1<q_2$ if $-2\leq E_1<E_2<0$, allows to conclude that $d\ell_1(E_1)<d\ell_1(E_2)$. Being this last result independent from the parameter $\lambda$ we can conclude that the length of the energy level $E_1$ is smaller than the one associated to the energy $E_2$.
\end{proof}

Establishing the loss of regularity of $\ell(E)$ at $E=0$ analytically is not a straightforward task, even if this property can be  easily appreciated by computing the graph of the function. Interestingly enough, 
there is also a loss of regularity of  the temporal LDs when crossing the separatrix as exemplified  on the right panel of  Fig.\,\ref{fig:leldpendulum}.  
Whilst the length of $\ell(E)$ reaches a global maximum at $E=0$, $\textrm{LD}(r_{0})$ is a local minimum on the line $(r_{0},\theta_{0}=0)$ when $r_{0}$ hits the separatrix at $r(\theta_{0})=2$. 
The temporal LDs have been computed over the time window $[0,20]$ (by approximating the flow with numerical solvers).  
We emphasize that the landscape presented in the left panel of Fig.\,\ref{fig:leldpendulum} depends solely on the energy, whilst the landscape presented in the right panel has required 
to compute the LDs over a specific direction (here by freezing the $\theta$ coordinate). The graph of $\ell(E)$ suggests that  the separatrix of the pendulum can be delineated and reconstructed semi-analytically, without the knowledge of the flow (or any approximation of it). 
In fact, each point $(r,\theta) \in \mathcal{D}$ defines a specific energy level $E$ for which $\ell(E)$ follows. 
Let us stress again that the
derivation of $\ell(E)$ does not involve the time parameter whatsoever.
The Fig.\,\ref{fig:MapsPendulum} presents the dynamical maps obtained on $\mathcal{D}=[-2.5,2.5]\times[-\pi,\pi]$, partitioned into a $500 \times 500$ Cartesian mesh of coordinates $(r,\theta)$. Let us observe
that this approach is redundant and used only to compare with the temporal
LD, indeed we could have used points located onto a line transverse to the flow
and then associate the computed value of the geometrical LD to all the points
lying on the same energy level.
From the computation of the geometrical LD on the nodes of this grid, we extracted also the
norm of the directional derivatives\footnote{
	Note that $B(r,\theta)$ is computed by using the 
	same mesh of initial conditions that has been used to compute $\ell(E)$ on each point $(r,\theta)$ of the grid.
	In particular,  to compute the norm of the directional derivatives, we dot not resample a new and more resolved mesh of initial conditions. 
	It is thus implicitly assumed that the resolution of the mesh is fine enough.   
}  with respect to the coordinates $(r,\theta)$
\begin{align}
B(r,\theta) = 
\norm{
	(\partial_{r}\ell(E), 
	\partial_{\theta}\ell(E))
}_{2}.
\end{align} 
Both dynamical maps succeed in revealing the separatrix.    

\begin{figure}
	\centering
	\includegraphics[width=1\linewidth]{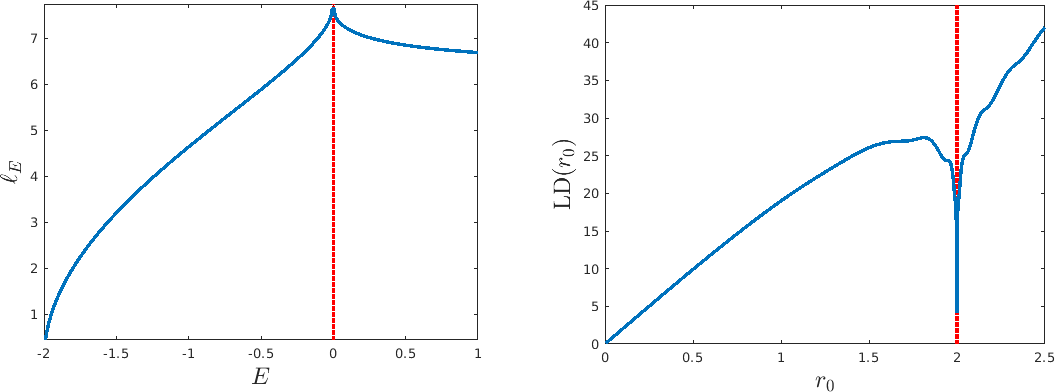}
	\caption{
		Landscapes of the geometrical and temporal LD on the pendulum model. 		
		(Left) The function $\ell$ as a function of $E$ spanning the librational and circulational domains.
		(Right) Landscape of the temporal LDs computed at $t=20$  
		along the line of initial condition 
		$(r_{0},\theta_{0}=0)$ where $r_{0}$ spans the librational and circulational domains. 
		In each case, the red dashed vertical lines indicate the location of the separatrix.   
	}
	\label{fig:leldpendulum}
\end{figure}

\begin{figure}
	\centering
	\includegraphics[width=1\linewidth]{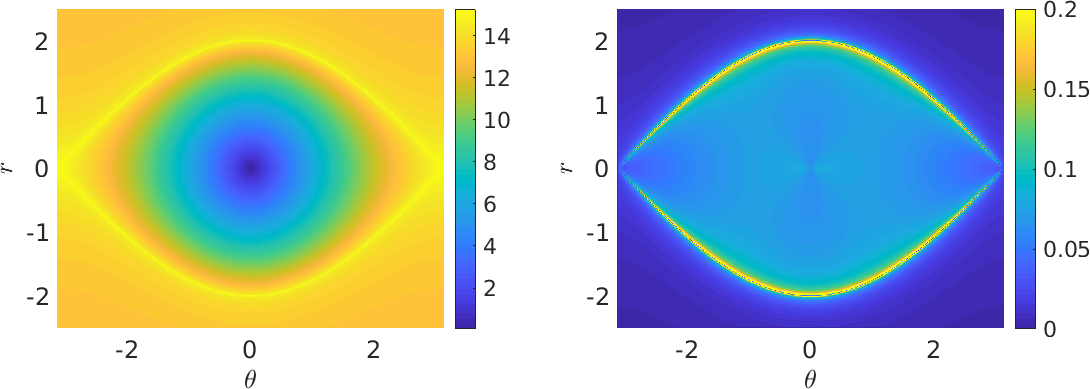}
	\caption{
		(Left)
		Dynamical map of the pendulum model obtained by computing $\ell(E)$ and 
		(Right)
		the norm of its gradient $B$ with respect to $(r,\theta)$.
	}
	\label{fig:MapsPendulum}
\end{figure}

\subsection{Rate of divergence of $\vert \dd \ell(E)/\dd E \vert$}\label{sub:dlede}   
In this subsection, we take advantage of the geometrical LD in order to  quantify the rate at which $\vert \dd \ell(E)/\dd E\vert$ becomes singular when $E \to E_{\textrm{sx}}=0$. 
We then connect and discuss the estimate to temporal LDs obtained under linearisation of the flow. \\

Even if we are interested in energy values close to $E=E_{\textrm{sx}}=0$, let us observe the following fact when
$E$ tends to the energy associated to the elliptic point,   
$E_{\textrm{ell.}}=-2$.

\begin{prop}\label{prop:SingElliptic}
	Let us consider $E \to -2^{+}$. Then we have 
	$\vert \dd \ell(E)/\dd E \vert \to + \infty $  
	at a rate $\mathcal{O}(1/\sqrt{E })$.
\end{prop}

\begin{proof}
	The linearised Hamiltonian of Eq.\,(\ref{eq:pendulum}) near the origin reads
	\begin{align}
	\mathcal{H} = \frac{r^2}{2} + \frac{\theta^{2}}{2}-2.
	\end{align}
	The 
	flow generates solutions which are circles, each circle is labelled by the value of the energy. The radius of the circles are   $\textrm{r}=\sqrt{2(E+2)}$. Therefore, $\ell(E)=2\pi \textrm{r} =2 \pi \sqrt{2(E+2)}$ 
	and
	$\dd \ell(E)/\dd E = 2\pi / \sqrt{2(E+2)}$. 
\end{proof}	

Prop.\,\ref{prop:SingElliptic} highlights that the geometrical LDs become singular when approaching the elliptic point. 
It turns out that the scaling of $\vert \dd \ell(E)/\dd E \vert$ as $\mathcal{O}(1/\sqrt{E})$ just found in the vicinity of the elliptic point is also valid locally near the saddle point and more generally when $E \to 0$. 

\begin{prop}\label{prop:SingHyperbolic}
	Let us consider the harmonic repulsor $\mathcal{H}=(r^2- \theta^{2})/2$.
	Then we have $\dd \ell(E)/\dd E  = \mathcal{O}(1/\sqrt{\vert E\vert})$.
\end{prop}

\begin{proof}
	Given the symmetries, we  focus on $r \ge 0, \theta \ge 0, E > 0$. For a given $E$, we interpret the finite branch of the hyperbola stemming from $(0,\sqrt{2E})$ 
	as the parametric curve $t \mapsto \big(\theta(t),r(t)\big)=(\sqrt{2E}\sinh t,\sqrt{2E}\cosh t)$, $t \in [0,t_{\star}]$. The length of this curve is given by 
	\begin{align}
	\ell(E)=\int_{0}^{t_{\star}} 
	\sqrt{\dot{\theta}^{2}(t)+\dot{r}^{2}(t)}\,\dd t
	=
	\sqrt{2E}\int_{0}^{t_{\star}}\sqrt{\sinh^2(t) + \cosh^{2}(t)}\, \dd t.
	\end{align}
	It follows that
	$\ell(E)= \sqrt{2E} \kappa(t_{\star})$
	and thus 
	\begin{align}
	\frac{\dd \ell(E)}{\dd E}= \frac{1}{\sqrt{E}}\frac{1}{\sqrt{2}}\kappa(t_{\star}).
	\end{align}     
\end{proof}

\begin{prop}\label{prop:SingHyp}
	Let us consider $E \to E_{\label{prop:SingHyp}\textrm{sx}}=0$. Then we have 
	$\vert \dd \ell(E)/\dd E \vert \to + \infty $  
	at a rate $\mathcal{O}(1/\sqrt{\vert E \vert})$.
\end{prop}

\begin{proof}(Semi-analytical)
	We numerically estimate $\ell(E)$ when $E \to 0$ both  from the librational ($E \to 0^{-}$) and circulational ($E \to 0^{+}$) domains.  Fig.\,\ref{fig:dledepend} compares the predicted estimates with the numerical estimations of 
	$\vert \dd \ell(E)/\dd E \vert$ and corroborate the claim. 
\end{proof}

\begin{figure}
	\centering
	\includegraphics[width=1\linewidth]{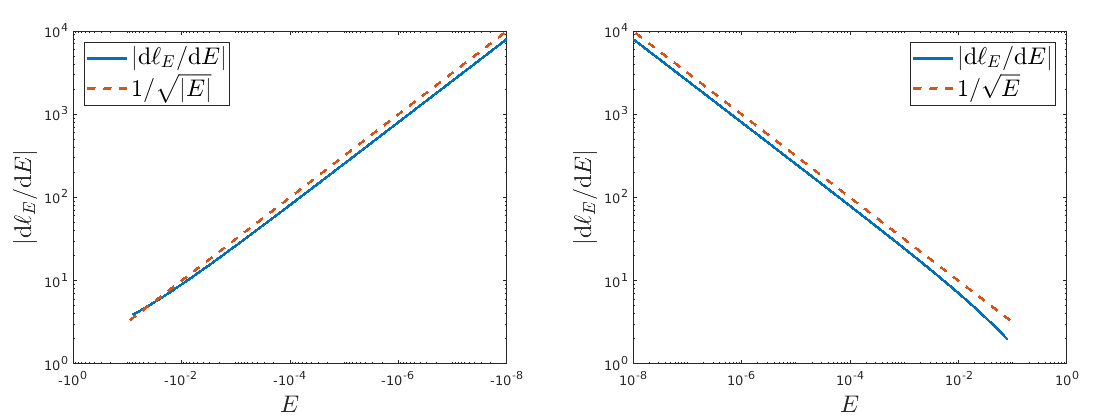}
	\caption{
		The divergence of $\dd \ell(E)/\dd E$ scales as $1/\sqrt{\vert E\vert}$
		when approaching the separatrix from 
		the librational (left) and circulational (right) domains. 
	}
	\label{fig:dledepend}
\end{figure}

We might wonder if  these power-law divergences are encapsulated into the temporal LDs.  
The following propositions show that a  
compatible estimate can be recovered near the elliptic and saddle points. 
\begin{prop}
	Consider the harmonic oscillator $\mathcal{H}=(r^{2}+\theta^{2})/2$. For $t>0$ fixed, we  have 
	\begin{align} 
	\frac{\dd}{\dd E}
	\rm{LD}\big((r,
	\theta);t\big) = \mathcal{O}(1/\sqrt{E}). 
	\end{align}       
\end{prop}	

\begin{proof}
	Following \cite{amMa13} (confer section $2.2.1$), given $(r,\theta)$ and $t>0$, straightforward computations give
	\begin{align}\label{eq:LDell}
	\rm{LD}\big((r,\theta);t\big)=2t\sqrt{r^{2}+\theta^{2}}=2t\sqrt{E}, 
	\end{align} 
	from which the announced equality follows. 
\end{proof}	

\begin{prop}
	Consider the harmonic repulsor  $\mathcal{H}=(r^{2}-\theta^{2})/2$. For $t>0$ fixed, we  have 
	\begin{align} 
	\frac{\dd}{\dd E}
	\rm{LD}\big((r,
	\theta);t\big) = \mathcal{O}(1/\sqrt{E}). 
	\end{align}       
\end{prop}	

\begin{proof}
	Given the Cauchy initial condition $(r_{0},\theta_{0})$ at time $t_{0}$, the solutions read
	\begin{align}\label{eq:FlowRepulsor}
	\left\{
	\begin{aligned}
	&r(t) = p_{0} \cosh(t)+q_{0}\sinh t, \\
	&\theta(t)= p_{0} \sinh(t)+q_{0}\cosh t.
	\end{aligned}
	\right.
	\end{align} 	
	Due to the symmetries, one might generically take initial conditions in the first quadrant and write  $(q_{0},p_{0})=(\sqrt{2E}\cosh t_{0},\sqrt{2E}\sinh t_{0})$. Then Eq.\,(\ref{eq:FlowRepulsor}) becomes
	\begin{align}
	\left\{
	\begin{aligned}
	&r(t) = \sqrt{2E} \cosh(t+t_{0}), \\
	&\theta(t)= \sqrt{2E} \sinh(t+t_{0}).
	\end{aligned}
	\right.
	\end{align} 
	We have
	\begin{align}
	\textrm{LD}(r_{0},\theta_{0};t)=
	\int_{0}^{t}
	\sqrt{\dot{r}^{2}(s)+\dot{\theta}^{2}(s)} \dd s
	=
	\sqrt{2E}\kappa(t)
	\end{align}
	from which follows the equality announced.
\end{proof}	
\section{Geometric Lagrangian descriptors for Duffing's oscillator}\label{sec:Duffing}
We repeat our previous steps on Duffing's Hamiltonian model  
\begin{align}
\mathcal{H}=\frac{y^{2}}{2}
- \frac{x^{2}}{2} + \frac{x^{4}}{4}, \, (x,y) \in \mathbb{R}^{2},
\end{align}
supporting an $8$-shaped separatrix. All trajectories are 
bounded, the phase space is shown in Fig.\,\ref{fig:PSDuffing}.  
The equilibrium $(x,y)=(0,0)$ is a saddle (with the corresponding energy $\mathcal{H}(0,0)=0$) and both equilibria $(x,y)=(0,\pm1)$ are elliptic. 
Given the symmetry of the phase space, if we restrict the domain to $x \ge 0$ and $y \ge 0$, we interpret the orbits as the planar curves, parametrised by $x$, defined by
\begin{align}\label{eq:PCDuffing}
y(x;E)=\frac{1}{\sqrt{2}} \sqrt{-x^{4}+2x^{2}+4E}. 
\end{align} 
By applying formula given in Eq.\,(\ref{eq:lE}) to Eq.\,(\ref{eq:PCDuffing}) leads to 
\begin{align}
\tilde{\ell}(E) = \int_{\mathcal{D}_{E}} 
\sqrt{1+\frac{2(x-x^{3})^{2}}{-x^4+2x^{2}+4E}} \, \dd y,
\end{align}
where 
\begin{align}
\left\{
\begin{aligned}
& \mathcal{D}_{E} = [0,(1+\sqrt{1+4E})^{1/2}], \,\textrm{if} \ E < 0,  \notag \\
& \mathcal{D}_{E} = [x_{1},x_{2}], \textrm{otherwise}
\end{aligned}
\right.
\end{align}
with
\begin{align}
x_{1}=\sqrt{1-\sqrt{1+4E}}, \, x_{2}=\sqrt{1+\sqrt{1+4E}}. 
\end{align}
With this set of formulas, we have $\ell(E)=4\tilde{\ell}(E)$.

The panel in Fig.\,\ref{fig:PanelDuffing} presents the observables considered in section \ref{sec:pendulum}. The length $\ell(E)$ is a local maximal in the vicinity of the separatrix $E=0$, where again $\ell(E)$ is not differentiable. This property underpins the ability to reconstruct the phase space via the length $\ell(E)$ and the $B$ map. 
Although we are not able to demonstrate analytically the form of the divergence, the semi-analytical apparatus provides evidence that $\vert \dd \ell(E)/\dd E\vert$ scales again as $\mathcal{O}(1/\sqrt{\vert E \vert})$ when $E \to 0$ where the temporal LD cannot be straightforwardly estimated. 

\begin{figure}
	\centering
	\includegraphics[width=0.6\linewidth]{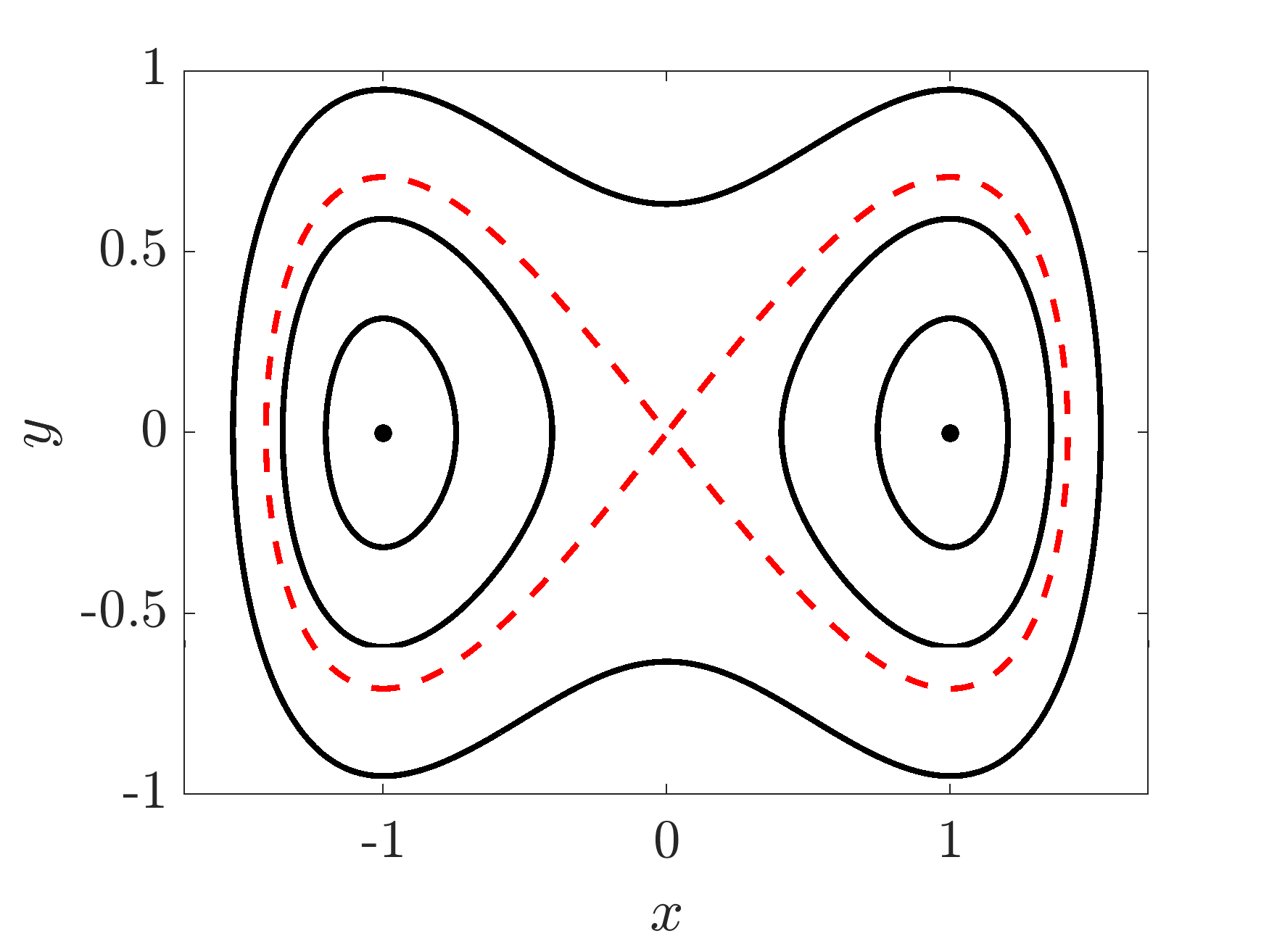}
	\caption{Phase space of Duffing's oscillator supporting the $8$-shaped separatrix.}
	\label{fig:PSDuffing}
\end{figure}

\begin{figure}
	\centering
	\includegraphics[width=1\linewidth]{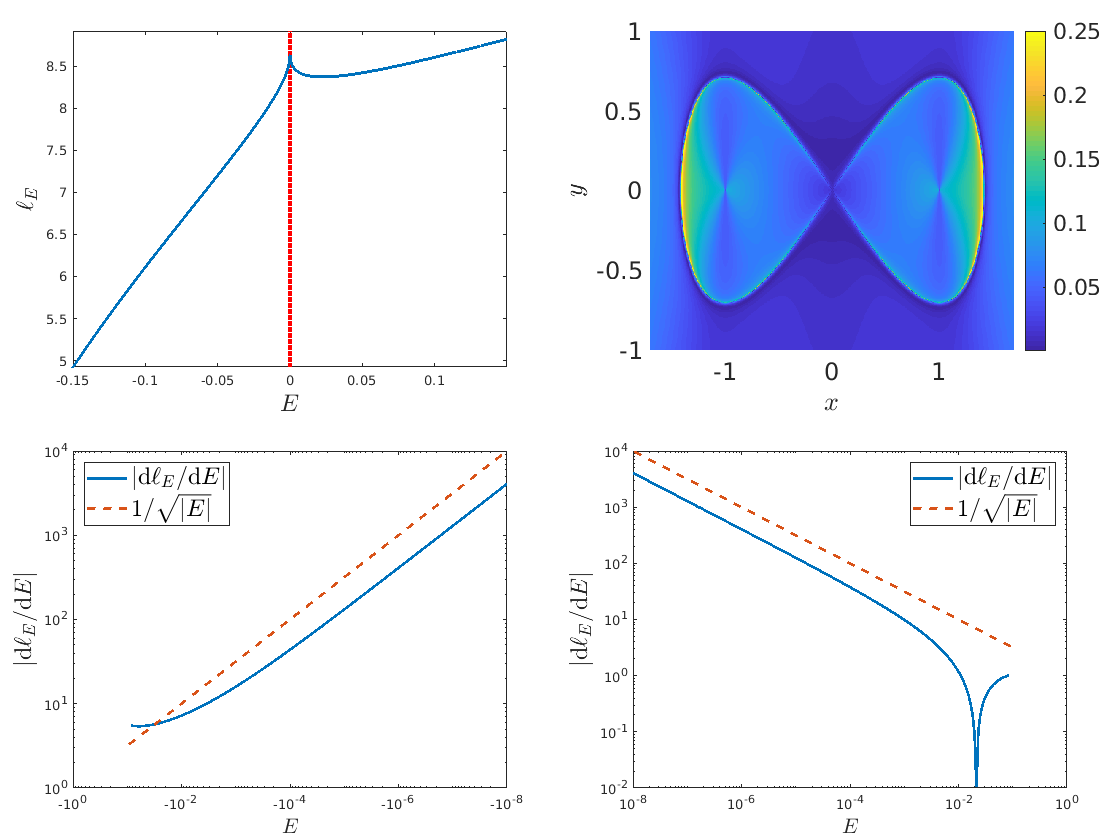}
	\caption{
		$\ell(E)$ as a function of $E$, $B$ map, and divergence laws associated to Duffing's oscillator.	
	}
	\label{fig:PanelDuffing}
\end{figure}

\section{Geometric Lagrangian descriptors for the fish-tail separatrix}\label{sec:FishTail}
We now turn our attention to the model 
\begin{align}\label{eq:HamFish}
\mathcal{H}=y^{2}+x^{3}+6x^{2}-32,  \, (x,y) \in \mathbb{R}^{2},
\end{align}
supporting a fish-tail separatrix and unbounded motions, as shown by the level sets of Fig.\,\ref{fig:PSFishTail}. 
The phase space contain two equilibrium, the unstable point $(x,y)=(-4,0)$ and the elliptic fixed point at $(x,y)=(0,0)$. Contrarily to the former models, all values of $E$ lead to $\ell(E) = +\infty$. 
This can be easily understood by observing that any energy
level, also the one associated to librations, contains an unbounded branch.
In order to deal with a finite value of $\ell(E)$, we first ``artificially'' bound the configuration space and allow the $x$ variable to a closed interval $\mathcal{I}=[a,b] \subset \mathbb{R}$. 
Given the symmetry of the phase space, for $x \in \mathcal{I}$ and $y \ge 0$, 
the level curves are interpreted as the planar curves 
\begin{align}
y(x;E) = \sqrt{-x^{3}-6x^{2}+E+32}. 
\end{align}  
Formula (\ref{eq:lE}) becomes
\begin{align}
\tilde{\ell}(E)=\int_{\mathcal{D}_{E}} 
\sqrt{1 + \frac{(-3x^{2}-12x)^{2}}{4(-x^{3}-6x^{2}+E+32)}} \, \dd y.
\end{align}
Let us now clarify the domain $\mathcal{D}_{E}$. For circulational orbits ($E>0$), we have
\begin{align}
\mathcal{D}_{E} = [x_{1},x_{2}],
\end{align}
where $x_{1}=a$ and $x_{2}=p^{1/3}+4p^{-1/3}-2$ where 
\begin{align}
p=\frac{1}{2}\sqrt{E(E+32)}+\frac{1}{2}(E+32)-8.
\end{align} 
For librational orbits, we have 
\begin{align}
\mathcal{D}_{E} = [x_{1},x_{2}] \cup [x_{3},x_{4}],
\end{align}
where $x_{1}=a$ and $\{x_{2},x_{3},x_{4}\}$ are the roots, ordered in increasing values, of the polynom 
$P \in \mathbb{R}^{3}[X]$ given by $P_{E}(X)=-X^{3}-6X^2+E+32$.
Under this set of formulas, we have $\ell(E)=2\tilde{\ell}(E)$.

The landscape $\ell(E)$ as a function of $E$, $B$ map and divergence laws are shown in Fig.\,\ref{fig:PanelFish}. 
For this particular set of computations, we considered $a=-5$. 
As for the pendulum and Duffing's models, $\ell(E)$ reaches a local maximum on the separatrix $E=0$ and forms a cusp point. The B map is thus able to delineate the separatrix.
By using the formula presented, we find again that the
divergence laws scale as $\mathcal{O}(1/\sqrt{\vert E\vert})$. Those observations rely a priori on the choice $a$, however, we have been able to reproduce them by considering 
a finite number of points $\{a_{i}\}_{i}$ such that $a_{i} \to -\infty$. Alternatively,  we can limit our analysis to the case of "bounded librations" with $E \le 0$ and $x \ge -4$ and study the finite lengths $\ell(E)$.  

\begin{figure}
	\centering
	\includegraphics[width=0.6\linewidth]{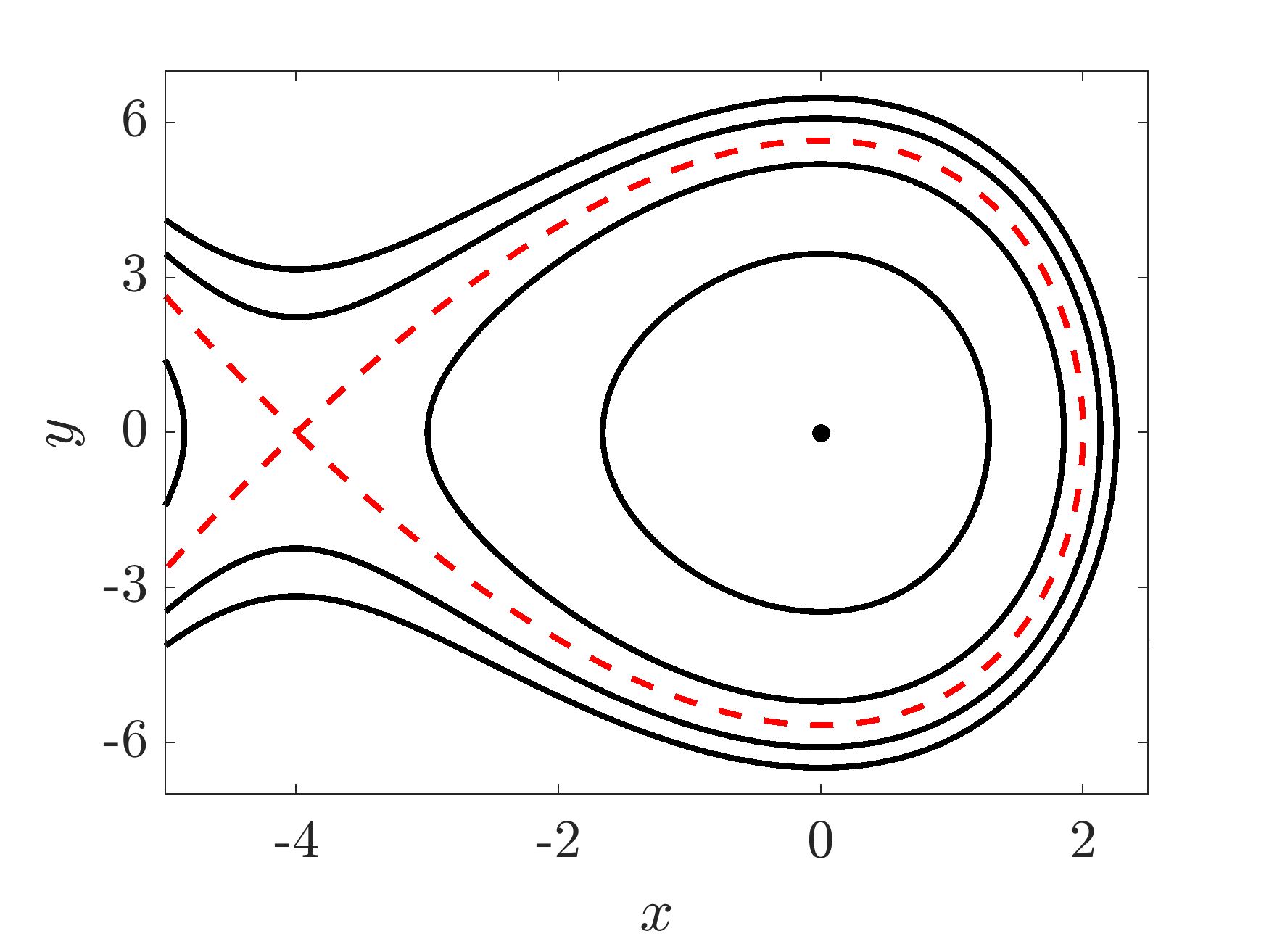}
	\caption{Fish-tail separatrix supporting unbounded motions.}
	\label{fig:PSFishTail}
\end{figure}

\begin{figure}
	\centering
	\includegraphics[width=1\linewidth]{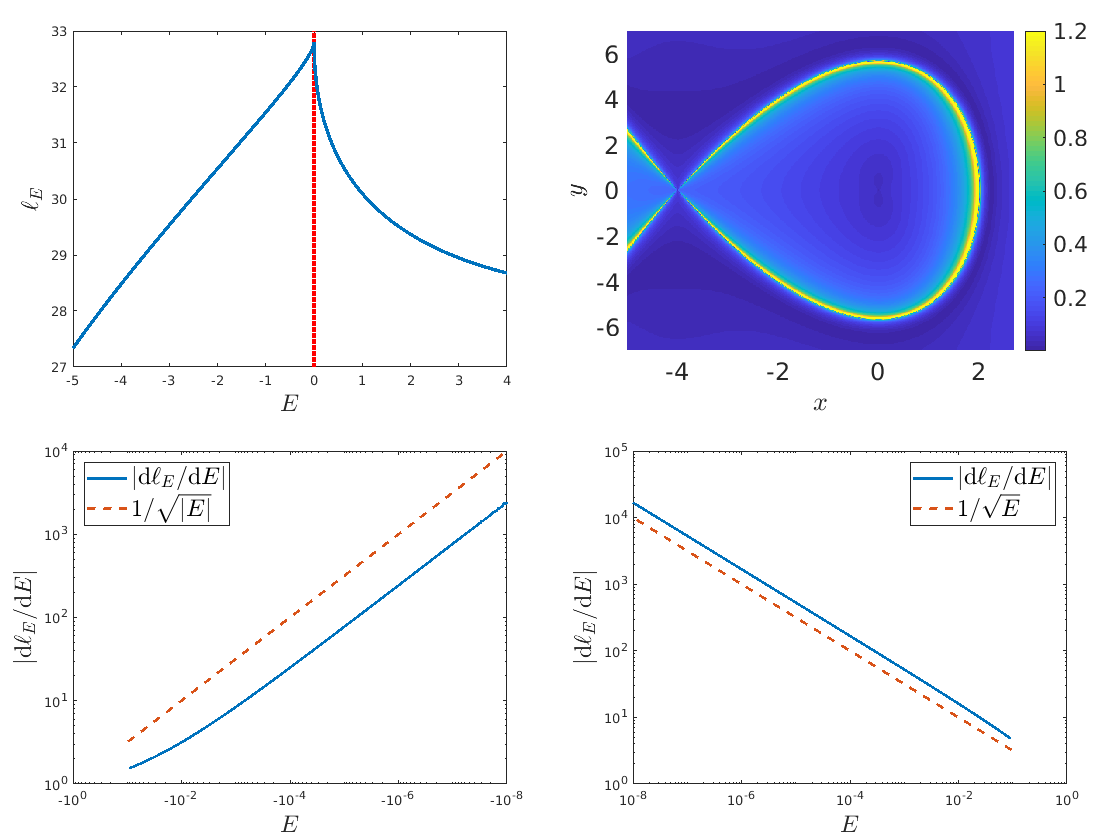}
	\caption{
		$\ell(E)$ as a function of $E$, $B$ map, and divergence laws associated to the model in Eq.\,(\ref{eq:HamFish}) supporting a fish-tail separatrix.	
	}
	\label{fig:PanelFish}
\end{figure}

\section{Conclusion}
Leveraging on the classical Lagrangian Descriptors for autonomous vector fields, this work has introduced a geometrical parametrisation for 1 degree-of-freedom system having the energy as first integral. 
The geometrical Lagrangian Descriptor corresponds to the length
$\ell(E)$ of the level curve labeled by the energy level $E$. 
It has a clear physical interpretation and is in particular free of the time variable. 
We have discussed applications of this framework on several and classical potential systems supporting the cat's eye, 8-shaped and fish-tail separatrices. 
For the models investigated, we have found the indicator to be a local maximum and cusp point on the energy labeling the separatrix. 
Alike the method of painting the energy integral over the phase space, 
the properties of the length metric permit to delineate successfully the separatrices of  dynamical systems. 
In addition to providing geometrical insights on the mechanisms driving Lagrangian Descriptors,     
the apparatus has  also conveniently allowed to  
characterise the rate at which $\vert \dd \ell(E)/\dd E \vert$ becomes singular near critical energies. 
Independent of the topology of the separatrices, we have revealed power-laws scaling as 
$\mathcal{O}(1/\sqrt{\vert E\vert})$.

\section*{acknowledgement}
	J.\,D. acknowledges funding from the ``Fonds de la Recherche Scientifique'' - FNRS and from the Australian Government through the round 9 Cooperative Research Centre Project ``A sensor network for integrated Space Traffic Management for Australia''.
	The authors acknowledge discussions with 
	Ana Maria Mancho, Makrina Agaoglou and Guillermo Garcia-Sanchez 
	that have ensued the 2nd Online Conference on Nonlinear Dynamics and Complexity, October $2021$. 

\bibliographystyle{apalike}
\bibliography{biblio}

\begin{thebibliography}{}

\bibitem[Carlo and Borondo, 2020]{cGa20}
Carlo, G.~G. and Borondo, F. (2020).
\newblock Lagrangian descriptors for open maps.
\newblock {\em Physical Review E}, 101(2):022208.

\bibitem[Carlo et~al., 2021]{cGa21}
Carlo, G.~G., Montes, J., and Borondo, F. (2021).
\newblock Lagrangian descriptors for the {B}unimovich stadium billiard.
\newblock {\em arXiv preprint arXiv:2110.03739}.

\bibitem[Coffey et~al., 1990]{cSh90}
Coffey, S., Deprit, A., Deprit, E., and Healy, L. (1990).
\newblock Painting the phase space portrait of an integrable dynamical system.
\newblock {\em Science}, 247(4944):833--836.

\bibitem[Craven et~al., 2017]{cGa17}
Craven, G.~T., Junginger, A., and Hernandez, R. (2017).
\newblock Lagrangian descriptors of driven chemical reaction manifolds.
\newblock {\em Physical Review E}, 96(2):022222.

\bibitem[Crossley et~al., 2021]{rCr21}
Crossley, R., Agaoglou, M., Katsanikas, M., and Wiggins, S. (2021).
\newblock From {P}oincar{\'e} maps to {L}agrangian descriptors: The case of the
  valley ridge inflection point potential.
\newblock {\em Regular and Chaotic Dynamics}, 26(2):147--164.

\bibitem[Laskar, 1993]{jLa93}
Laskar, J. (1993).
\newblock Frequency analysis of a dynamical system.
\newblock {\em Celestial Mechanics and Dynamical Astronomy}, 56(1):191--196.

\bibitem[Lopesino et~al., 2015]{cLo15}
Lopesino, C., Balibrea, F., Wiggins, S., and Mancho, A.~M. (2015).
\newblock Lagrangian descriptors for two dimensional, area preserving,
  autonomous and nonautonomous maps.
\newblock {\em Communications in Nonlinear Science and Numerical Simulation},
  27(1-3):40--51.

\bibitem[Lopesino et~al., 2017]{cLo17}
Lopesino, C., Balibrea-Iniesta, F., Garc{\'\i}a-Garrido, V.~J., Wiggins, S.,
  and Mancho, A.~M. (2017).
\newblock A theoretical framework for {L}agrangian descriptors.
\newblock {\em International Journal of Bifurcation and Chaos}, 27(01):1730001.

\bibitem[Madrid and Mancho, 2009]{jaMa09}
Madrid, J.~J. and Mancho, A.~M. (2009).
\newblock Distinguished trajectories in time dependent vector fields.
\newblock {\em Chaos: {A}n {I}nterdisciplinary {J}ournal of {N}onlinear
  {S}cience}, 19(1):013111.

\bibitem[Mancho et~al., 2013]{amMa13}
Mancho, A.~M., Wiggins, S., Curbelo, J., and Mendoza, C. (2013).
\newblock Lagrangian descriptors: {A} method for revealing phase space
  structures of general time dependent dynamical systems.
\newblock {\em Communications in Nonlinear Science and Numerical Simulation},
  18(12):3530--3557.

\bibitem[Mendoza et~al., 2014]{cMe14}
Mendoza, C., Mancho, A., and Wiggins, S. (2014).
\newblock Lagrangian descriptors and the assessment of the predictive capacity
  of oceanic data sets.
\newblock {\em Nonlinear Processes in Geophysics}, 21(3):677--689.

\bibitem[Mendoza and Mancho, 2010]{cMe10}
Mendoza, C. and Mancho, A.~M. (2010).
\newblock Hidden geometry of ocean flows.
\newblock {\em Physical review letters}, 105(3):038501.

\bibitem[Naik et~al., 2019]{nSh19}
Naik, S., Garc{\'\i}a-Garrido, V.~J., and Wiggins, S. (2019).
\newblock Finding {NHIM}: {I}dentifying high dimensional phase space structures
  in reaction dynamics using lagrangian descriptors.
\newblock {\em Communications in Nonlinear Science and Numerical Simulation},
  79:104907.

\bibitem[Revuelta et~al., 2019]{fRe19}
Revuelta, F., Benito, R., and Borondo, F. (2019).
\newblock Unveiling the chaotic structure in phase space of molecular systems
  using {L}agrangian descriptors.
\newblock {\em Physical Review E}, 99(3):032221.

\bibitem[Skokos, 2010]{chSk10}
Skokos, C. (2010).
\newblock The {L}yapunov characteristic exponents and their computation.
\newblock In {\em Dynamics of {S}mall {S}olar {S}ystem {B}odies and
  {E}xoplanets}, pages 63--135. Springer.

\end{thebibliography}

\end{document}